\documentclass[a4paper]{amsart}

\usepackage{amsthm,amsfonts,amsmath, amssymb}

\usepackage{isomath}
\usepackage{lmodern} 
\usepackage[italic]{mathastext}

\usepackage{setspace}
\usepackage{graphicx,float,enumerate,subfigure,tikz}
\usepackage{algpseudocode}
\usepackage{verbatim}
\usepackage{url}
\usepackage{epstopdf}
\usepackage[capitalise]{cleveref}
\newtheorem{theorem}{Theorem}
\newtheorem{corollary}[theorem]{Corollary}
\newtheorem{lemma}[theorem]{Lemma}
\newtheorem{proposition}[theorem]{Proposition}

\newenvironment{claimproof}{\noindent\textit{Proof.}}{\hfill$\square$}

\theoremstyle{definition}
\newtheorem{definition}[theorem]{Definition}

\newtheorem{example}[theorem]{Example}

\theoremstyle{remark} 
\newtheorem{remark}[theorem]{Remark}

\crefname{remark}{Remark}{Remarks}

\theoremstyle{definition}
\newtheorem{claim}{Claim}
\usepackage{etoolbox}
\AtEndEnvironment{proof}{\setcounter{claim}{0}}

\theoremstyle{remark}
\newtheorem{case}{Case}
\usepackage{etoolbox}
\AtEndEnvironment{proof}{\setcounter{case}{0}}
\AtEndEnvironment{claimproof}{\setcounter{case}{0}}

\usepackage{mathtools}

\DeclarePairedDelimiter\floor{\lfloor}{\rfloor}
\DeclareMathOperator{\dist}{dist}

\DeclareMathOperator{\lk}{link}

\newcommand{\R}{\mathbb{R}}

\crefname{rmk}{Remark}{Remarks}
\crefname{problem}{Problem}{Problems}
\crefname{example}{Example}{Examples}
\date{\today}
\title{The linkedness of cubical polytopes: The cube}

\usepackage{amsaddr}

\author{Hoa T. Bui}
\address{Centre for Informatics and Applied Optimisation, Federation University Australia\\
Faculty of Science and Engineering, Curtin University, Australia}
\email{\texttt{hoa.bui@curtin.edu.au}}
\author{Guillermo Pineda-Villavicencio \& Julien Ugon}
\address{Centre for Informatics and Applied Optimisation, Federation University Australia\\School of Information Technology, Deakin University, Australia}
\email{\texttt{julien.ugon@deakin.edu.au}} 
\email{\texttt{work@guillermo.com.au}}

\thanks{Julien Ugon's research was partially supported by ARC discovery project DP180100602.}

\begin{document}
\begin{abstract}

The paper is concerned with the linkedness of the graphs of cubical polytopes. A graph with at least $2k$ vertices is \textit{$k$-linked} if, for every set of $k$ disjoint pairs of vertices, there are $k$ vertex-disjoint paths joining the vertices in the pairs. We say that a polytope is \textit{$k$-linked} if its graph is $k$-linked.

We establish that the $d$-dimensional cube
 is $\floor{(d+1)/2}$-linked, for every $d\ne 3$; this is the maximum possible linkedness of a $d$-polytope. This result implies that, for every $d\ge 1$, a cubical $d$-polytope is  $\floor{d/2}$-linked, which answers a question of Wotzlaw \cite{Ron09}.

Finally, we introduce the notion of strong linkedness, which is slightly stronger than that of linkedness. A graph $G$ is {\it strongly $k$-linked} if it has at least $2k+1$ vertices and, for  every vertex $v$ of $G$, the subgraph $G-v$ is $k$-linked. We show that cubical 4-polytopes are strongly $2$-linked and that, for each $d\ge 1$,  $d$-dimensional cubes  are strongly $\floor{d/2}$-linked. 

\end{abstract}
\maketitle

\section{Introduction}

A (convex) polytope is the convex hull of a finite set $X$ of points in $\R^{d}$; the \textit{convex hull} of $X$ is  the smallest convex set containing $X$. The \textit{dimension} of a polytope in $\R^{d}$ is one less than the maximum number of affinely independent points in the polytope. A polytope of dimension $d$ is referred to as a \textit{$d$-polytope}. 

 A {\it face} of a polytope $P$ in $\R^{d}$ is $P$ itself, or  the intersection of $P$ with a hyperplane in $\R^{d}$ that contains $P$ in one of its closed halfspaces. Faces other than $P$ are polytopes of smaller dimension. A face of dimension 0, 1, and $d-1$ in a $d$-polytope is a \textit{vertex}, an {\it edge}, and a {\it facet}, respectively.  The {\it graph} $G(P)$ of a polytope $P$ is the undirected graph formed by the vertices and edges of the polytope. 
 
This paper studies the linkedness of {\it cubical $d$-polytopes}, $d$-dimensional polytopes with all their facets being cubes. A \textit{$d$-dimensional cube} is the convex hull in $\R^{d}$ of the $2^{d}$ vectors $(\pm 1,\ldots,\pm 1)$. By a cube we mean any polytope that is \textit{combinatorially equivalent} to a cube; that is, one whose face lattice is isomorphic to the face lattice of a cube. 
 
Denote by $V(X)$ the vertex set of a graph or a polytope $X$. Given sets $A,B$ of vertices in a graph, a path from $A$ to $B$, called an {\it $A-B$ path}, is a (vertex-edge) path $L:=u_{0}\ldots u_{n}$ in the graph such that $V(L)\cap A=\{u_{0}\}$  and $V(L)\cap B=\{u_{n}\}$. We write $a-B$ path instead of $\{a\}-B$ path, and likewise, write $A-b$ path instead of $A-\{b\}$. 

Let $G$ be a graph and $X$ a subset of $2k$ distinct vertices of $G$. The elements of $X$ are called {\it terminals}. Let $Y:=\{\{s_{1},t_{1}\}, \ldots,\{s_{k},t_{k}\}\}$ be an arbitrary labelling and (unordered) pairing of all the vertices in $X$. We say that $Y$ is {\it linked} in $G$ if we can find disjoint $s_{i}-t_{i}$ paths for $i\in [1,k]$, where $[1,k]$ denotes the interval $1,\ldots,k$. The set $X$ is {\it linked} in $G$ if every such pairing of its vertices is linked in $G$. Throughout this paper, by a set of disjoint paths, we mean a set of vertex-disjoint paths. If $G$ has at least $2k$ vertices and every set of exactly $2k$ vertices is linked in $G$, we say that $G$ is {\it $k$-linked}. If the graph of a polytope is $k$-linked we say that the polytope is also {\it $k$-linked}.

Unless otherwise stated, we use the graph theoretical notation and terminology from \cite{Die05}, while the polytope theoretical notation and terminology from \cite{Zie95}. Moreover, when referring to graph-theoretical properties of a polytope such as minimum degree, linkedness and connectivity, we mean properties of its graph.

Linkedness is an attractive property of graphs. Being $k$-linked imposes a stronger demand on a graph than just being $k$-connected. Let $G$ be a graph with at least $2k$ vertices, and  let $S:=\{s_{1},\ldots,s_{k}\}$ and $T:=\{t_{1},\ldots,t_{k}\}$ be two disjoint $k$-element sets of vertices in $G$. It follows that, if $G$ is $k$-connected then the sets $S$ and $T$ can be joined \textbf{setwise}  by disjoint paths (namely, by $k$ disjoint $S-T$ paths); this is a consequence of Menger's theorem (\cref{thm:Menger-consequence}). And if $G$ is $k$-linked then the sets can be joined \textbf{pointwise} by disjoint paths.  

From a structural point of view, linkedness guarantees the existence of many subdivisions in a graph.  A graph $Y$ is a \textit{subdivision} of a graph $X$ if it can be obtained from $X$ by subdividing edges of $X$. The definition of $k$-linkedness yields that, if a graph is $k$-linked, then it has a subdivision of every graph on $k$ edges.

From an algorithmic point of view, linkedness is closely related to the classical \textit{disjoint paths problem} \cite{RobSey-XIII}: given a graph $G$ and a set  $Y:=\{\{s_{1},t_{1}\}, \ldots,\{s_{k},t_{k}\}\}$ of $k$ pairs of terminals in $G$, decide whether or not $Y$ is linked in $G$. A natural optimisation version of this problem is to find the largest subset of the pairs so that there exist disjoint paths connecting the selected pairs. The disjoint paths problem has found many applications in the field of transportation networks and computer science in general \cite{Kaw11,KawKobRee12}. It is a special case of a multicommodity flow problem where there exist $k$ different commodities that need to go from  the sources $s_1,s_2,\ldots,s_k$ to $t_1,t_2,\ldots,t_k$; for information on multicommodity flows consult \cite{KorVyg18}, and for further information on the disjoint paths problem consult \cite{KawKobRee12} and the references therein.

All the 2-linked graphs have been characterised \cite{Sey80,Tho80}. In the context of polytopes, one consequence is that, with the exception of simplicial 3-polytopes, no 3-polytope is 2-linked; a {\it simplicial polytope} is one in which every facet is a simplex.  Another consequence is that every 4-polytope is 2-linked. We provide new proofs of these two results: \cref{cor:nonsimplicial-3polytope,prop:4polytopes}.

There is a linear function $f(k)$ such that every $f(k)$-connected graph is $k$-linked, which follows from works of Bollob\'as and Thomason \cite{BolTho96}; Kawarabayashi, Kostochka, and Yu \cite{KawKosYu06}; and Thomas and Wollan \cite{ThoWol05}. In the case of polytopes, Larman and Mani \cite[Thm.~2]{LarMan70} proved that every $d$-polytope  is $\floor{(d+1)/3}$-linked, a result that was slightly improved to $\floor{(d+2)/3}$  in \cite[Thm.~2.2]{WerWot11}.  There are $d$-polytopes that are $d$-connected but not $(d+1)$-connected, and so $\floor{(d+1)/2}$ is an upper bound for the linkedness of $d$-polytopes.

Apart from the work of Larman and Mani \cite{LarMan70}, the study of linkedness in graphs of polytopes has been motivated by a problem in the first edition of the Handbook of Discrete and Computational Geometry  \cite[Problem 17.2.6]{GooORo97-1st}. The problem asked whether or not every $d$-polytope is $\floor{d/2}$-linked. This question had already been answered in the negative by Gallivan \cite{Gal85} in the 1970s with a construction of a $d$-polytope that is not $\floor{2(d+4)/5}$-linked. A weak positive result however follows from \cite{ThoWol05}: every $d$-polytope with minimum degree at least $5d$ is $\floor{d/2}$-linked.

In view of \cite[Problem 17.2.6]{GooORo97-1st} and the negative result of Gallivan \cite{Gal85}, researchers have focused efforts on finding families of $d$-polytopes that are $\floor{d/2}$-linked. Simplicial $d$-polytopes are $\floor{(d+1)/2}$-linked, for every $d\ge 2$ \cite[Thm.~2]{LarMan70}. 
In his PhD thesis \cite[Question~5.4.12]{Ron09}, Wotzlaw asked whether every cubical $d$-polytope is $\floor{d/2}$-linked. In \cref{thm:weak-linkedness-cubical} we answer his question in the affirmative  by establishing that a $d$-cube is $\floor{(d+1)/2}$-linked, for every $d\ne 3$ (see \cref{thm:cube}).  We remark 
that the linkedness of the cube was first established in \cite[Prop.~4.4]{Mes16} as part of a study of linkedness in Cartesian products of graphs, but no self-contained proof was available. 

In a subsequent paper \cite{BuiPinUgo20a}, we prove a stronger result: a cubical $d$-polytope is $\floor{(d+1)/2}$-linked, for every $d\ne 3$; which is best possible. In anticipation of this result, in \cref{prop:link-cubical} we prove that certain cubical $d$-polytopes that are embedded in the $(d+1)$-cube are $\floor{(d+1)/2}$-linked, for every $d\ge 3$.

Let  $X$ be a set of vertices in a graph $G$. Denote by $G[X]$ the subgraph of $G$ induced by $X$, the subgraph of $G$ that contains all the edges of $G$ with vertices in $X$. Write $G-X$ for $G[V(G)\setminus X]$. If $X=\{v\}$, then we write $G-v$ instead of $G-\{v\}$.

Finally, we introduce the notion of strong linkedness, a property marginally stronger than linkedness.
We say that a graph $G$ is {\it strongly $k$-linked} if it has at least $2k+1$ vertices and, for every vertex $v$ of $G$, the subgraph $G-v$ is $k$-linked. We show that cubical 4-polytopes are strongly $2$-linked and that, for each  $d\ge 1$,  $d$-dimensional cubes  are strongly $\floor{d/2}$-linked.

\section{Preliminary results}
\label{sec:preliminary}
This section groups a number of results that will be used in later sections of the paper.

\cref{prop:3-polytopes,prop:4polytopes} follow from the characterisation of 2-linked graphs carried out in \cite{Sey80,Tho80}. Both propositions also have proofs stemming from arguments in the form of \cref{lem:linear-paths}; for the sake of completeness we give such proofs.

\begin{lemma}[{\cite[Thm.~3.1]{Sal67}}] \label{lem:linear-paths} Let $P$ be a $d$-polytope, and let $f$ be a linear function on $\mathbb{R}^{d}$ satisfying $f(x)>0$ for some $x\in P$. If $u$ and $v$ are vertices of $P$ with $f(u)\ge 0$ and $f(v)\ge 0$, then there exists a $u-v$ path $x_{0}x_{1}\ldots x_{n}$ with $x_{0}=u$ and $x_{n}=v$ such that $f(x_{i})>0$ for each $i\in [1,n-1]$.  
\end{lemma}

We state Balinski's theorem on the connectivity of polytopes.
 
\begin{theorem}[Balinski {\cite{Bal61}}]\label{thm:Balinski} For every $d\ge 1$, the graph of a $d$-polytope is $d$-connected. \end{theorem}
 
A path in the graph is called {\it $X$-valid} if no inner vertex of the path is in $X$. The {\it distance} between two vertices $s$ and $t$ in a graph $G$, denoted $\dist_{G}(s,t)$, is the length of a shortest path between the vertices.  
  
 \begin{definition}[Configuration 3F]\label{def:Conf-3F} Let $X$ be a set of at least four terminals in a 3-cube and let $Y$ be a  labelling and pairing of the vertices in $X$. A terminal of $X$, say $s_{1}$, is in {\it Configuration 3F} if the following conditions are satisfied:
 \begin{enumerate}
 \item[(i)] four vertices of $X$ appear in a 2-face $F$ of the cube;
 \item[(ii)] the terminals in the pair $\left\{s_{1},t_{1}\right\}\in Y$ are at distance two in $F$; and 
\item[(iii)] the neighbours of $t_{1}$ in $F$ are all vertices of $X$.
 \end{enumerate}	
\end{definition}  
 
 Configuration 3F is the only configuration in a 3-cube that prevents the linkedness of a  pairing $Y$ of four vertices, as  \cref{prop:3-polytopes} attests. A sequence $a_{1},\ldots, a_{n}$ of vertices in a cycle is in {\it cyclic order} if, while traversing the cycle, the sequence appears in clockwise or counterclockwise order. It follows that, if pairing $Y:=\left\{\left\{s_{1},t_{1}\right\},\left\{s_{2},t_{2}\right\}\right\}$ of vertices in a 3-cube appears in cyclic order $s_{1},s_{2},t_{1},t_{2}$ in a 2-face, then all the vertices in $Y$ are in Configuration 3F.

\begin{proposition}\label{prop:3-polytopes} Let $G$ be the graph of a 3-polytope and let $X$ be a set of four vertices of $G$. The set $X$ is linked in $G$ if and only if there is no facet of the polytope containing all the vertices of $X$.
\end{proposition}
\begin{proof} Let $P$ be a 3-polytope embedded in $\R^{3}$ and let  $X$ be an arbitrary set of four vertices in $G$. We first establish the necessary condition by proving the contrapositive. Let $F$ be a 2-face containing the vertices of $X$ and  consider a planar embedding of $G$ in which $F$ is the outer face. Label the vertices of $X$ so that they appear in the cyclic order $s_{1},s_{2},t_{1},t_{2}$. Then the paths $s_{1}-t_{1}$ and $s_{2}-t_{2}$ in $G$ must inevitably intersect, implying that $X$ is not linked.

Assume there  is no 2-face of $P$ containing all the vertices of $X$. Let $H$ be a (linear) hyperplane that contains $s_{1}$, $s_{2}$ and $t_{1}$, and let $f$ be a linear function that vanishes on $H$ (this may require a translation of the polytope). Without loss of generality, assume that $f(x)>0$ for some $x\in P$ and that $f(t_{2})\ge 0$. 

First consider the case that $H$ is a supporting hyperplane of a 2-face $F$.  The subgraph $G(F)-\{s_{2}\}$ is connected by Balinski's theorem (\cref{thm:Balinski}), and so there is an $X$-valid $L_{1}:=s_{1}-t_{1}$ path on  $G(F)$. Then, use \cref{lem:linear-paths} to find an $L_{2}:=s_{2}-t_{2}$ path in which each inner vertex has positive $f$-value. The paths $L_{1}$ and $L_{2}$ are clearly disjoint. 

Now consider the case that $H$ intersects the interior of $P$. Then there is a vertex in $P$ with $f$-value  greater than zero and a vertex with $f$-value less than zero. Use \cref{lem:linear-paths} to find  an $s_{1}-t_{1}$ path in which each inner vertex has negative $f$-value and an $s_{2}-t_{2}$ path in which each inner vertex has positive $f$-value. 
\end{proof}

The subsequent corollary follows at once from \cref{prop:3-polytopes}. 
\begin{corollary} No nonsimplicial 3-polytope is 2-linked.
\label{cor:nonsimplicial-3polytope}
\end{corollary}

 The same reasoning employed in the proof of the sufficient condition of \cref{prop:3-polytopes} settles \cref{prop:4polytopes}.

\begin{proposition}[2-linkedness of 4-polytopes]\label{prop:4polytopes} Every 4-polytope is 2-linked.
\end{proposition}
\begin{proof} Let $G$ be the graph of a 4-polytope embedded in $\mathbb R^{4}$. Let $X$ be a given set of four vertices in $G$ and let $Y:=\{\{s_{1},s_{2}\},\{t_{1},t_{2}\}\}$ a labelling and pairing of the vertices in $X$. 

Consider a linear function $f$ that vanishes on a linear hyperplane $H$ passing through $X$. Consider the two cases in which either $H$ is  a supporting hyperplane of a facet $F$ of $P$ or $H$ intersects the interior of $P$.

Suppose $H$ is a supporting hyperplane of a facet $F$. First, find an $s_{1}-t_{1}$ path in the subgraph $G(F)-\{s_{2},t_{2}\}$, which is connected by Balinski's theorem (\cref{thm:Balinski}). Second, use \cref{lem:linear-paths} to find an $s_{2}-t_{2}$ path that touches $F$ only at $\{s_{2},t_{2}\}$.     

If instead $H$ intersects the interior of $P$ then there is a vertex in $P$ with $f$-value  greater than zero and a vertex with $f$-value less than zero. Use \cref{lem:linear-paths} to find  an $s_{1}-t_{1}$ path in which each inner vertex has negative $f$-value and an $s_{2}-t_{2}$ path in which each inner vertex has positive $f$-value. \end{proof}

\section{$d$-cube}  
\label{sec:cube}

Consider the $d$-cube $Q_{d}$. Let   $v$ be a vertex in $Q_{d}$ and let $v^{o}$ denote the vertex at distance $d$ from $v$, called the vertex {\it opposite} to $v$. Besides, denote  by $F^o$ the facet disjoint from a facet $F$ of $Q_{d}$; we say that $F$ and $F^{o}$ is a pair of {\it opposite} facets. 

\begin{definition}[Projection $\pi$]\label{def:projection}
	For a pair of opposite facets $\{F,F^{o}\}$ of $Q_{d}$, define a projection $\pi^{Q_{d}}_{F^{o}}$ from $Q_{d}$ to $F^{o}$ by sending a vertex $x\in F$ to the unique neighbour $x^{p}_{F^o}$ of $x$ in $F^{o}$, and  a vertex $x\in F^{o}$ to itself (that is, $\pi^{Q_{d}}_{F^{o}}(x)=x$); write $\pi^{Q_{d}}_{F^o}(x)=x^{p}_{F^o}$ to be precise, or write $\pi(x)$ or $x^{p}$ if the cube $Q_{d}$ and the facet $F^o$ are understood from the context. 
\end{definition}

\Cref{def:projection} is exemplified in \cref{fig:3cube-projection}(a). 
We extend this projection to sets of vertices: given a pair $\{F,F^{o}\}$ of opposite facets and a set $X\subseteq V(F)$, the projection $X^{p}_{F^{o}}$ or $\pi^{Q_{d}}_{F^{o}}(X)$ of $X$ onto $F^{o}$ is the set of the projections  of the vertices in $X$ onto $F^{o}$. For an $i$-face $J\subseteq F$, the projection $J^{p}_{F^{o}}$ or $\pi_{F^{o}}^{Q_{d}}(J)$ of $J$ onto $F^{o}$ is the $i$-face consisting of the projections of all the vertices of $J$ onto $F^{o}$. For a pair $\{F,F^{o}\}$ of opposite facets in $Q^{d}$, the restrictions of the projection $\pi_{F^{o}}$ to $F$  and the projection $\pi_{F}$ to $F^{o}$ are bijections.

Let  $Z$ be a set of vertices in the graph of a $d$-cube $Q_{d}$. If, for some pair of opposite facets $\{F,F^o\}$, the set $Z$ contains both a vertex $z\in V (F)\cap Z$ and its projection $z_{F^{o}}^p\in V (F^o)\cap Z$, we say that the pair $\{F,F^o\}$ is {\it associated} with the set $Z$ in $Q_{d}$ and that $\{z,z^p\}$ is an {\it associating pair}.  Note that an associating pair can associate only one pair of opposite facets. See \cref{fig:3cube-projection}(b)-(c). 

The next lemma lies at the core of our methodology. 
 
 \begin{figure}     
\includegraphics{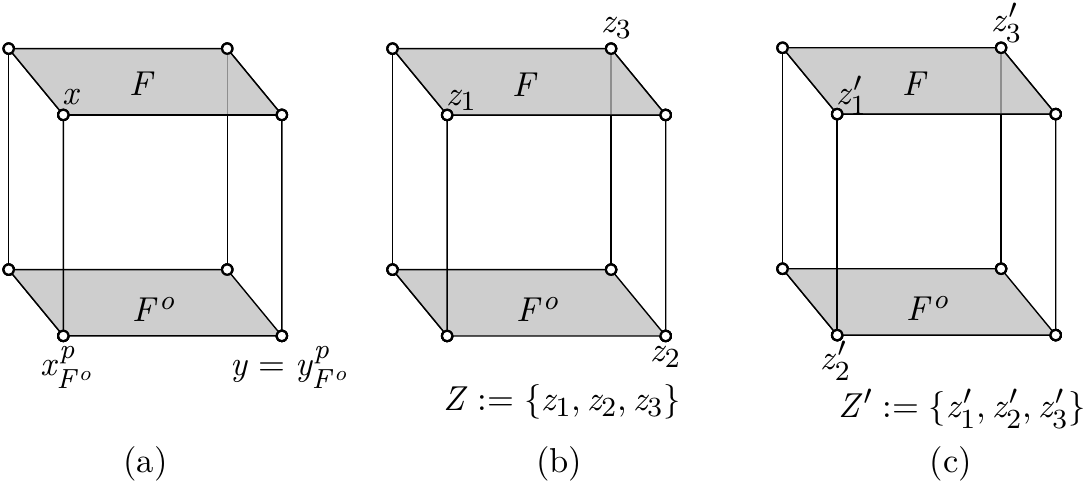}
\caption{The 3-cube with a pair $\{F,F^o\}$ of opposite facets highlighted.  (a) Examples of the projection $\pi_{F^o}$ for the pair $\{F,F^o\}$. (b) The pair $\{F,F^O\}$ is not associated with the set $Z:=\{z_1,z_2,z_3\}$. (c)  The pair $\{F,F^O\}$ is associated with the set $Z':=\{z'_1,z'_2,z'_3\}$, and $\{z'_1,z'_2\}$ is an associating pair.} \label{fig:3cube-projection} 
\end{figure}

\begin{lemma}\label{lem:facets-association} Let $Z$ be a nonempty subset of $V(Q_{d})$. Then the number of pairs $\{F,F^o\}$ of opposite facets associated with $Z$ is at most $|Z|-1$. 
\end{lemma} 

\begin{proof}
Let $G:=G(Q_{d})$ and let $Z\subset V(Q_{d})$ with $|Z|\ge 1$ be given. Consider a pair $\{F,F^{o}\}$ of opposite facets. Define a {\it direction} in the cube as the set of the $2^{d-1}$ edges between $F$ and $F^{o}$; each direction corresponds to a pair of opposite facets.  The $d$ directions partition the edges of the cube into sets of cardinality $2^{d-1}$.  (The notion of direction stems from thinking of the cube as a zonotope \cite[Sec.~7.3]{Zie95}) 

A pair of facets is associated with the set $Z$ if and only if the subgraph $G[Z]$ of $G$ induced by $Z$ contains an edge from the corresponding direction.

If a direction is present in a cycle $C$ of $Q_{d}$, then the cycle contains at least two edges from this direction. Indeed, take an edge $e=uv$  on $C$ that belongs to a direction between a pair $\{F,F^{o}\}$ of opposite facets.   After traversing  the edge $e$ from $u\in V(F)$ to $v\in V(F^{o})$, for the cycle to come back to the facet $F$, it must contain another edge from the same direction.   Hence, by repeatedly removing edges from cycles in $G[Z]$ we obtain a spanning forest of  $G[Z]$ that contains an edge for every direction present in $G[Z]$. As a consequence, the number of such directions is at most the number of edges in the forest, which is upper bounded by $|Z|-1$. (A {\it forest} is a graph with no cycles.) \end{proof}

The relevance of the lemma stems from the fact that a pair of opposite facets $\{F,F^{o}\}$ not associated with a given set of vertices $Z$ allows each vertex $z$ in $Z$ to have ``free projection''; that is, for every $z\in Z\cap V(F)$ the projection $\pi_{F^o}(z)$ is not in $Z$, and for $z\in Z\cap V(F^{o})$ the projection $\pi_{F}(z)$ is not in $Z$.

\section{Connectivity of the $d$-cube}

We next unveil some further properties of the cube that will be used in subsequent sections.

Given sets $A,B,X$ of vertices in a graph $G$,  the set  $X$  {\it separates} $A$ from $B$ if every $A-B$ path in the graph contains a vertex from $X$. A set $X$ separates two vertices $a,b$ not in $X$ if it separates $\{a\}$ from $\{b\}$. We call the set $X$ a {\it separator} of the graph. We will also require the following three assertions. 
 
\begin{proposition}[{\cite[Prop.~1]{Ram04}}]\label{prop:cube-cutsets} Any separator $X$ of cardinality $d$ in $Q_{d}$ consists of the $d$ neighbours of some vertex in the cube.
\end{proposition}

A set of vertices in a graph is {\it independent} if no two of its elements are adjacent. Since there are no triangles in a $d$-cube,  \cref{prop:cube-cutsets} gives at once the following corollary.
 
 \begin{corollary}\label{cor:separator-independent}  A separator of cardinality $d$ in a $d$-cube is an independent set.
 \end{corollary}

\begin{remark}\label{rmk:cubical-common-neighbours} If $x$ and $y$ are vertices of a cube, then they share at most two neighbours. In other words, the complete bipartite graph $K_{2,3}$ is not a subgraph of the cube; in fact, it is not an induced subgraph of any simple polytope \cite[Cor.~1.12(iii)]{PfePilSan12}. 
\end{remark}

\section{Linkedness of the $d$-cube} In this section, we establish the linkedness of $Q_{d}$  (\cref{thm:cube}). We make heavy use of Menger's theorem \cite[Thm.~3.3.1]{Die05} henceforth, and so we remind the reader of the theorem and one of one of its consequences.

\begin{theorem}[Menger's theorem, {\cite[Sec.~3.3]{Die05}}]\label{thm:Menger} Let $G$ be a  graph, and let $A$ and $B$ be two subsets of its vertices. Then the minimum number of vertices separating $A$ from $B$ in $G$ equals the maximum number of disjoint $A-B$ paths in $G$. 
\end{theorem} 

\begin{theorem}[Consequence of Menger's theorem]\label{thm:Menger-consequence} Let $G$ be a $k$-connected graph, and let $A$ and $B$ be two subsets of its vertices, each of cardinality at least $k$. Then there are $k$ disjoint $A-B$ paths in $G$. 
\end{theorem}

Two vertex-edge paths are {\it independent} if they share no inner vertex. 

\begin{lemma} Let P be a cubical $d$-polytope with $d\ge 4$. Let $X$ be a set of $d+1$  vertices in $P$, all contained in a facet $F$. Let $k:=\floor{(d+1)/2}$. Arbitrarily label and pair $2k$ vertices in $X$ to obtain $Y:=\{\{s_{1},t_{1}\},\ldots,\{s_{k},t_{k}\}\}$. Then, for at least $k-1$ of these pairs $\{s_{i},t_{i}\}$, there is an  $X$-valid  $s_{i}-t_{i}$ path in $F$.   
\label{lem:short-distance} 
\end{lemma}
 
\begin{proof} If, for each pair in $Y$ there is an  $X$-valid path in $F$ connecting the pair, we are done. So assume there is a pair in $Y$, say $\{s_{1},t_{1}\}$, for which an $X$-valid $s_{1}-t_{1}$ path does not exist in $F$. Since $F$ is $(d-1)$-connected, there are $d-1$ independent $s_{1}-t_{1}$ paths (\cref{thm:Menger-consequence}), each containing a vertex from $X\setminus \{s_{1},t_{1}\}$; that is, the set $X\setminus \{s_{1},t_{1}\}$, with cardinality $d-1$, separates $s_{1}$ from $t_{1}$ in $F$. By \cref{prop:cube-cutsets}, the vertices in $X\setminus \{s_{1},t_{1}\}$ are the neighbours of $s_{1}$ or $t_{1}$ in $F$, say of $s_{1}$. 

Take any pair in $Y \setminus \{\{s_{1},t_{1}\}\}$, say $\{s_{2},t_{2}\}$. Observe that $s_2$ and $t_2$ are both neighbours of $s_1$. If there was no $X$-valid $s_{2}-t_{2}$ path in $F$, then,  by \cref{prop:cube-cutsets}, the set $X\setminus \{s_{2},t_{2}\}$ would separate $s_{2}$ from $t_{2}$ and would consist of the neighbours of $s_{2}$ or $t_{2}$ in $F$, say of $s_{2}$. But in this case, a vertex $x$ in $X\setminus \{s_{1},s_{2},t_{1},t_{2}\}$, which exists since $|X|\ge 5$, would form a triangle with $s_{1}$ and $s_{2}$, a contradiction. See also \cref{cor:separator-independent}. Since our choice of $\{s_{2},t_{2}\}$ was arbitrary, we must have an $X$-valid path in $F$ between any pair $\{s_{i},t_{i}\}$ for $i\in [2,k]$.          
\end{proof}

 For a set $Y:=\{\{s_{1},t_{1}\},\ldots,\{s_{k},t_{k}\}\}$ of pairs of vertices in a graph, a {\it $Y$-linkage} $\{L_{1},\ldots,L_{k}\}$ is a set of disjoint paths with the path $L_{i}$ joining the pair $\{s_{i},t_{i}\}$ for $i\in [1,k]$.  For a path $L:=u_{0}\ldots u_{n}$ we often write $u_{i}Lu_{j}$ for $0\le i\le j\le n$  to denote the subpath $u_{i}\ldots u_{j}$. We are now ready to prove \cref{thm:cube}.

The definition of $k$-linkedness gives the following lemma at once.

\begin{lemma}\label{lem:k-linked-def} Let $\ell\le k$. Let $X^*$ be a set of $2\ell$ distinct vertices of a $k$-linked graph $G$,  let $Y^*$ be a labelling and pairing of the vertices in $X^*$, and let $Z^*$ be a set of at most $2k-2\ell$ vertices in $G$ such that $X^*\cap Z^*=\emptyset$. Then there exists a $Y^*$-linkage in $G$ that avoids every vertex in $Z^*$.  
\end{lemma}
 
 We require a result on strong linkedness. With \cref{prop:4polytopes,lem:short-distance} at hand, we can verify  that cubical 4-polytopes are strongly $2$-linked. 
   
\begin{theorem}[Strong linkedness of cubical 4-polytopes]\label{thm:4polytopes-strong-linkedness}   Every cubical 4-polytope is strongly $2$-linked.
\end{theorem} 
\begin{proof}   

Let $G$ denote the graph of a cubical 4-polytope $P$ embedded in $\mathbb R^{4}$. Let $X$ be a set of five vertices in $G$.  Arbitrarily pair four vertices of $X$ to obtain $Y:=\{\{s_{1},t_{1}\},\{s_{2},t_{2}\}\}$. Let $x$ be the vertex of $X$ not being paired in $Y$. We aim to find two disjoint paths $L_{1}:=s_{1}-t_{1}$ and  $L_{2}:=s_{2}-t_{2}$ such that each path $L_{i}$ avoids the vertex $x$. The proof is very similar to that of \cref{prop:3-polytopes,prop:4polytopes}. 
 
Consider a linear function $f$ that vanishes on a linear hyperplane $H$ passing through $\{s_{1},s_{2},t_{1},x\}$.  Assume that $f(y)>0$  for some $y\in P$ and that $f(t_{2})\ge 0$. 

Suppose first that $H$ is  a supporting hyperplane of a facet $F$ of $P$. If $t_{2}\not\in V(F)$ (that is, $f(t_2)>0$), then find an $X$-valid $L_{1}:=s_{1}-t_{1}$ path in $F$ using the 3-connectivity of $F$ (Balinski's theorem). Then use \cref{lem:linear-paths}  to find an $X$-valid $s_{2}-t_{2}$ path in which each inner vertex has positive $f$-value. If instead  $t_{2}\in F$, then $X\subset V(F)$ and \cref{lem:short-distance} ensures the existence of an $X$-valid $s_{i}-t_{i}$ path in $F$ for some $i=1,2$, say for $i=1$. Then use \cref{lem:linear-paths} to find an $X$-valid $s_{2}-t_{2}$ path in which each inner vertex has positive $f$-value. 

So assume that $H$ intersects the interior of $P$. Then there is a vertex in $P$ with $f$-value greater than zero and a vertex with $f$-value less than zero. In this case, use \cref{lem:linear-paths} to find  an $X$-valid $s_{1}-t_{1}$ path in which each inner vertex has negative $f$-value and an $X$-valid $s_{2}-t_{2}$ path in which each inner vertex has positive $f$-value. 
\end{proof}

 Not every 4-polytope is strongly $2$-linked. Take a two-fold pyramid $P$ over a quadrangle $Q$. Then $P$ is a 4-polytope on six vertices, say $s_{1},s_{2},t_{1},t_{2},x,y$. Let the sequence $s_{1},s_{2},t_{1},t_{2}$ appears in $Q$ in cyclic order, and let the vertex $x$ be in $V(P)\setminus V(Q)$. To see that $P$ is not strongly 2-linked, observe that, for every two paths $s_{1}-t_{1}$ and $s_{2}-t_{2}$ in $P$, they intersect or one of them contains $x$.   
 
 We continue with a simple lemma from \cite[Sec.~3]{WerWot11}.
 
\begin{lemma}[{\cite[Sec.~3]{WerWot11}}]\label{lem:k-linked-subgraph} Let $G$ be a $2k$-connected graph and let $G'$ be a $k$-linked subgraph of $G$. Then $G$ is $k$-linked.
\end{lemma}
 
 We are now ready to establish the linkedness of the $d$-cube.
 
\begin{theorem}[{Linkedness of the cube}]\label{thm:cube}  For every $d\ne 3$, a $d$-cube is $\floor{(d+1)/2}$-linked.
\end{theorem}
\begin{proof} The cases of $d=1,2$ are trivially true. For the remaining values of $d$, we proceed by induction, with the base $d=4$ given by \cref{prop:4polytopes}. So assume that $d\ge 5$.
 
Let $k:=\floor{(d+1)/2}$,  then $2k-1\le d$.  Let $X$ be any set of $2k$ vertices, our terminals, in the graph $G$ of the  $d$-cube $Q_{d}$ and let $Y:=\{\{s_{1},t_{1}\},\ldots,\{s_{k},t_{k}\}\}$ be a pairing and labelling of the vertices of $X$. We aim to find a $Y$-linkage $\{L_{1},\ldots,L_{k}\}$ with $L_{i}$ joining the pair $\{s_{i},t_{i}\}$ for $i=1,\ldots,k$. 

We first deal with the case of even $d\ge 6$. In this setting, $d=2k$, and so $G$ is $2k$-connected by Balinski's theorem. Furthermore, by the induction hypothesis, the graph $G'$ of every facet of $Q_d$, a $(d-1)$-polytope, is $\floor{d/2}$-linked, namely $k$-linked. \cref{lem:k-linked-subgraph} now ensures that $G$ is $k$-linked. As a consequence, for the rest of the proof, we focus on the case of odd $d\ge 5$.

For a facet $F$ of $Q_{d}$, let $F^{o}$ denote the facet opposite to $F$. 
\begin{figure}     
\includegraphics{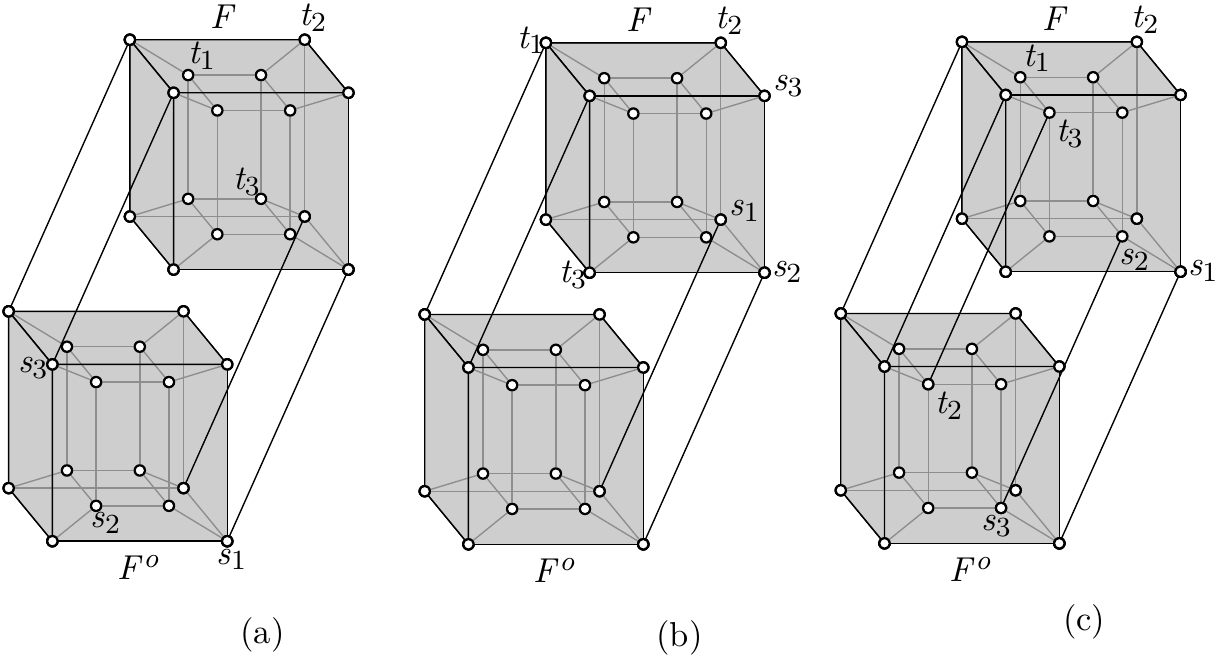}
\caption{The three scenarios of \cref{thm:cube}, exemplified in $Q_5$. Two opposite facets $\{F, F^{o}\}$ of $Q_5$ are highlighted in grey. There is an edge between each vertex in $F$ and the ``copy" of the vertex in $F^{o}$; while some edges from $F$ to $F^{o}$  are missing, we have depicted enough edges to show how the other edges should be drawn.   (a) First scenario: every pair in $Y$ is at distance $5$. (b) Second scenario: every vertex in $X$ lies in the facet $F$ of $Q_5$. (c) Third scenario: not every pair in $Y$ is at distance $5$, namely $s_1$ and $t_1$, and not every vertex in $X$ lies in some facet  of $Q_5$, in this case $\dist_{Q_5}(s_1,t_2)=5$ and so no facet can contain both of them (\cref{rmk:cube-1}).}\label{fig:5cube-scenarios} 
\end{figure}

%We consider three scenarios: (1)  all the pairs in $Y$ lie in some facet of $Q_{d}$, (2)  a pair  of $Y$ lies in some facet $F$ of $Q_{d}$ but not every vertex of $X$ is in $F$, and (3) no pair of $Y$ lies in a facet of $Q_{d}$, which amounts to saying that every pair in $Y$ is at distance $d$ in $Q_{d}$. For the sake of readability, each scenario is highlighted in bold.

We consider three scenarios: (1) every pair in $Y$ lies at distance $d$ in $Q_{d}$; (2) all vertices in $X$ lie in a facet $F$ of $Q_{d}$; and (3) the remaining case: not every pair in $Y$ lies at distance $d$ in $Q_{d}$ and not all vertices in $X$ lie in the same facet of $Q_d$.  The three scenarios are depicted in \cref{fig:5cube-scenarios}. It is helpful to have the following remark  at hand.
\begin{remark}
\label{rmk:cube-1}
Two vertices $x$ and $y$ lie in some facet of $Q_d$ if and only if $\dist (x,y)<d$.
\end{remark}

{\bf In the first scenario every pair in $Y$ lies at distance $d$.} From \cref{lem:facets-association} it follows that there exists a pair $\{F,F^o\}$ of opposite facets of $Q_{d}$ that is not associated with $X_{s_{1}}:=X\setminus \{s_{1}\}$, since $|X\setminus \{s_{1}\}|\le d$ and there are $d$  pairs of the form $\{F,F^o\}$. This means that for every $x\in X_{s_{1}}\cap V(F)$, $\pi_{F^o}(x)\not \in X_{s_{1}}$ and that for every $x\in X_{s_{1}}\cap V(F^o)$, $\pi_{F}(x)\not \in X_{s_{1}}$.  Because every vertex in $G$ is in either $F$ or $F^o$, we must have that, if some $s_i$ is in one of $\{F,F^o\}$, then $t_i$ must be in the other (\cref{rmk:cube-1}). Thus, without loss of generality, we can assume that $s_{1},\ldots,s_{k}\in F^o$ and $t_{1},\ldots,t_{k}\in F$. 

From $s_1\not \in F$  it now follows that $\pi_{F}(s_i)\not \in X$, for each $i\in [2,k]$. Besides, since the pair $\{F,F^o\}$ is not associated with $X_{s_{1}}$, we have that  \[\pi_{F^o}(t_1)\not \in X_{s_1},   \pi_{F^o}(t_2)\not \in X_{s_1},     \ldots, \pi_{F^o}(t_k)\not \in X_{s_1}.\]
It is the case that $\pi_{F^o}(t_1)\ne s_1$, otherwise  $s_1$ and $t_1$ would be adjacent, contradicting the fact that $\dist(s_1,t_1)=d\ge 5$. Hence   $\pi_{F^o}(t_1)\not \in X$. 
Because $k>2$, it is also true  that $\pi_{F^{o}}(t_{i})\ne s_{1}$ for some $t_{i}\in V(F)$ with $i\in[2,k]$, say $\pi_{F^{o}}(t_{2})\ne s_{1}$. Then  $\pi_{F^{o}}(t_{2})\not\in X$. We summarise our discussion below.
\begin{align}
\label{eq:cube-2}
      \pi_{F^o}(t_1)\not \in X,\;    \pi_{F^o}(t_2)\not \in X,\;     \pi_{F^o}(t_3)\not \in X_{s_1},\ldots, \pi_{F^o}(t_k)\not \in X_{s_1},   \\
          \pi_{F}(s_2)\not \in X,\;    \pi_{F}(s_3)\not \in X,    \ldots, \pi_{F}(s_k)\not \in X.\nonumber
\end{align}

Let $X':=\{\pi_{F}(s_{3}),\ldots,\pi_{F}(s_{k}),t_{3},\ldots,t_{k}\}$. By the induction hypothesis, $F$ is $(k-1)$-linked. In the notation of \cref{lem:k-linked-def}, if we let $\ell:=k-2$, $X^*:=X'$, $Y^*:=\{\{\pi_{F}(s_{3}),t_3\},\ldots,\{\pi_{F}(s_{k}),t_{k}\}\}$, and $Z^*:=\left\{t_{1},t_{2}\right\}$, then we can find $k-2$ disjoint  paths $L_{i}'$ in $F$ between $\pi_{F}(s_{i})$ and $t_{i}$  for $i\in [3,k]$, with each path avoiding $\left\{t_{1},t_{2}\right\}$; here $|Z^*|=2(k-1)-2\ell=2$ (\cref{lem:k-linked-def}). Let $L_{i}:=s_{i}\pi_{F}(s_{i})L_{i}'t_{i}$, for each $i\in [3,k]$. 

Now we find the paths $L_1$ and $L_2$ in $F^o$. The induction hypothesis yields that $F^{o}$ is $(k-1)$-linked. We also have that  $\pi_{F^o}(t_1)\not \in X$ and $\pi_{F^o}(t_2)\not \in X$, according to \eqref{eq:cube-2}. In the notation of \cref{lem:k-linked-def}, we let $\ell:=2$, $X^*:=\{s_1,\pi_{F^o}(t_1),s_2,\pi_{F^o}(t_2)\}$, $Y^*:=\{(s_{1},\pi_{F^o}(t_1)),(s_{2},\pi_{F^o}(t_2))\}$, and $Z^*:=\left\{s_{3},\ldots,s_{k}\right\}$. Then, according to \cref{lem:k-linked-def}, for $k\ge 4$ we can find disjoint paths $L_{1}':=s_{1}-\pi_{F^{o}}(t_{1})$ and $L_{2}':=s_{2}-\pi_{F^{o}}(t_{2})$ in $F^{o}$, each avoiding the set $\left\{s_{3},\ldots,s_{k}\right\}$, since $|Z^*|=k-2\le 2(k-1)-2 \ell=2k-6$ for $k\ge 4$.

The case $k=3$ requires special attention. In this setting, $Z^*=\{s_3\}$, $d=5$, and $\dim F^o=4$.  Since every cubical 4-polytope is strongly 2-linked by \cref{thm:4polytopes-strong-linkedness},  $F^o$ is strongly 2-linked. Hence, the strong 2-linkedness of $F^o$ now gives the existence of disjoint paths $L_{1}':=s_{1}-\pi_{F^{o}}(t_{1})$ and $L_{2}':=s_{2}-\pi_{F^{o}}(t_{2})$ in $F^{o}$, each  avoiding $Z^*$. 

As a consequence, we let $L_1:=s_1L_1'\pi_{F^o}(t_1)t_1$ and $L_2:=s_2L_2'\pi_{F^o}(t_2)t_2$. In this way, we have found a $Y$-linkage $\{L_{1},\ldots,L_{k}\}$ with $L_{i}$ joining the pair $\{s_{i},t_{i}\}$ for $i=1,\ldots,k$. This completes the proof of the first scenario.

{\bf In the second scenario all vertices in $X$ lie in a facet $F$ of $Q_{d}$}. In this case, \cref{lem:short-distance} gives an $X$-valid path $L_{1}$ in $F$ joining a pair in $Y$, say $\{s_{1},t_{1}\}$. The projection in $Q_{d}$ of every vertex in $(X\setminus \{s_{1},t_{1}\})\cap V(F)$ onto $F^o$ is not in $X$. Define $Y^{p}:=\{\{s_{2}^{p},t_{2}^{p}\},\ldots,\{s^{p}_{k},t^{p}_{k}\}\}$ as the set of $k-1$ pairs of projections of the corresponding vertices in $Y\setminus\left\{\{s_{1},t_{1}\}\right\}$ onto $F^{o}$. By the induction hypothesis on $F^o$, $F^o$ is $(k-1)$-linked, and so there is a $Y^{p}$-linkage $\{L_{2}^{p},\ldots, L_{k}^{p}\}$ with $L_{i}^{p}:=s_{i}^{p}-t_{i}^{p}$ for $i\in [2,k]$. Each path $L_{i}^{p}$ can be extended with $s_{i}$ and $t_{i}$ to obtain a path $L_{i}:=s_{i}-t_{i}$ for $i\in [2,k]$. And together, all the paths $\{L_{1},\ldots, L_{k}\}$ give the desired $Y$-linkage in the cube. 

{\bf Finally, let us move to the third scenario: not every pair in $Y$ lies at distance $d$ in $Q_{d}$ and not all vertices in $X$ lie in the same facet.} It follows that some pair  in $Y$, say $\{s_{1},t_{1}\}$,  lies in some facet $F$ of $Q_{d}$  (\cref{rmk:cube-1}), but not every vertex in $X$ is in $F$. Let $N_{K}(x)$ denote the set of neighbours of a vertex $x$ in a face $K$ of the cube and let $N(x)$ denote the set of all the neighbours of $x$ in the cube. 

We define a function $\rho:X\to X$  that maps $x\in X$ to the terminal with which it is paired in $Y$: $\{x,\rho(x)\}\in Y$.  Let $X_F := (X\setminus\{s_1,t_1\})\cap V(F)$.  We define the following sets.
\begin{align*}
    Y_\alpha&:= \{\{x,\rho(x)\}\in Y: x\in X_F,\,\rho(x)\in N_F(x)\},\\
    X_\alpha& := \{x\in X_{F}: \{x,\rho(x)\}\in Y_\alpha\},\\
    X_\beta& := X_F\setminus X_\alpha.
\end{align*}
%$X_{0}:=\{x\in X_{F}: \,\{x,y\}\in Y \text{ and } y\in N_{F}(x)\}$ and
We construct the desired $Y$-linkage $\{L_1,\ldots,L_k\}$ according to the following cases:
\begin{description}
    \item[Case (i)] each pair $\{s_j,t_j\}$ in $Y_\alpha$ is joined by the path $L_j:=s_jt_j$ in $F$;
    \item[Case (ii)] each pair $\{s_j,t_j\}$ in  $Y\setminus(Y_\alpha\cup\{s_1,t_1\})$ will be joined using the induction hypothesis on the $(d-1)$-cube $F^o$, the facet opposite to $F$; 
    \item[Case (iii)] the pair $\{s_1,t_1\}$ will be joined by an $X$-valid path in $F$.  
\end{description}

Case (i) is done, so we focus on Case (ii). We will need to project some terminals from $F$ to $F^o$, namely the ones in $X_\beta$.  However, unlike {\bf Scenario 2}, it may happen that the projection $\pi_{F^o}$ of a terminal vertex $x\in X_\beta$ onto $F^o$ is also a terminal vertex, which will cause the problem of having some paths intersect. In order to use the projection $\pi_{F^o}$, we define an injective map $\omega:X_\beta\to V(F)$ so that, for each  $x\in X_\beta$, we have that either $\omega(x)=x$ or $\omega(x)\in N_F(x)$ and that  $\{\omega(x),\pi_{F^o}(\omega(x))\} \cap (X\setminus \{x,\rho(x)\}) = \emptyset$. 

The motivation for the  map $\omega$ is to define a path $M_x = x\omega(x)\pi_{F^o}(\omega(x))$, of length at most 2, from each $x\in X_\beta$ to $F^{o}$ so that the vertices $\pi_{F^o}(\omega(x))$ and $\pi_{F^o}(\omega(\rho(x)))$ can be joined in $F^{o}$ by an $X$-valid path. \Cref{ex:cube-linkedness} and \cref{fig:5cube-paths}  illustrate the function $\omega$.  The construction of $\omega$ follows a general remark.
\begin{remark}
\label{rmk:cube-3}
 Whenever possible we set $\omega(x)=x$. 
 Only when the projection $\pi_{F^o}$ of $x$, namely $x^p_{F^o}$, is in  $X\setminus \{\rho(x)\}$, do we set $\omega(x)\in N_F(x)\setminus X_F$.
\end{remark}
 Lemma~\ref{lem:xvalid-path-existence} shows that the injective map $\omega$ exists.

\begin{lemma}\label{lem:xvalid-path-existence} There exists an injective map $\omega:X_\beta\to V(F)$ such that, for each $x\in X_\beta$,
\begin{equation}\label{eq:injectiveProjection}
\{\omega(x),\pi_{F^o}(\omega(x))\} \cap (X\setminus \{x,\rho(x)\}) = \emptyset.
\end{equation}
\end{lemma}

\begin{proof} 
Let $X'$ be the maximal subset of $X_\beta$ such that an injective map $\omega$ exists and satisfies Condition~\eqref{eq:injectiveProjection}. We will prove that $X'=X_\beta$ by contradiction. Assume that $X'\neq X_\beta$ and let $x\in X_\beta\setminus X'$. Then 
\begin{equation}
\label{eq:cube-4}
  x^p_{F^o} \in X\setminus \{\rho(x)\},  
\end{equation}
otherwise setting  $\omega(x) = x$ would satisfy \eqref{eq:injectiveProjection},   extending the injection $\omega$ to $X'\cup \{x\}$.

Let us define the set $O_{x}$ as the subset of vertices $v$ in $N_F(x)$ that \emph{cannot} be selected as $v=\omega(x)$, because  they violate either Condition~\eqref{eq:injectiveProjection} or the injectivity of $\omega$.

For a vertex $v$ to violate Condition~\eqref{eq:injectiveProjection}, it must be that either $v\in X$ (say $v$ is of \textit{type 1}), or  $v\not\in X$ and $\pi_{F^o}(v) \in X\setminus \{\rho(x)\}$ (say $v$ is \textit{type 2}) (equivalently this means that there is a terminal $z\in V(F^o)\cap (X\setminus \{\rho(x)\})$ such that $v=\pi_F(z)$).  For a vertex $v$ to violate the injectivity of $\omega$, there must exist $z\in X'$ such that $v=\omega(z)$; we remark that a vertex of type 1 or 2 could violate the injectivity of $\omega$. As a consequence, we say that $v$ is of \textit{type 3} if it violates the injectivity but it is not of type 1 or 2.  

Therefore the set $O_{x}$ can be defined as the  subset of $N_F(x)$ that violate Condition~\eqref{eq:injectiveProjection} or the injectivity of $\omega$.  
$$O_{x} = N_F(x)\cap \left( X \cup \{z^p_F: z\in X\setminus \{\rho(x)\}\} \cup \{\omega(z) : z\in X'\} \right).$$

%$$O_{x} = N_F(x)\cap \left( X \cup \{z^p_F: z^p_F\notin X,z\in X\setminus \{\rho(x)\}\} \cup \{\omega(z) : z\in X'\} \right).$$
%$$O_{x} = \bigcup_{z\in X'} \left(\omega(z) \cap N_F(x)\right)\; \bigcup \left(X\cap N_F(x)\right).$$
%The set $O_{x}$ consists of three types of neighbours $v$ of $x$ in $F$: 

\begin{remark}\label{rmk:type3}
For type 3 vertices, note that $z\not \in O_{x}$, otherwise $z,w(z)\in N_F(x)$, implying that the vertices $z$, $x$, and $\omega(z)$ would all be pairwise neighbours but there are no triangles in $Q_{d}$.
\end{remark}

Note that $|N_{F}(x)|= d-1$. Thus, to show that there is  a  suitable vertex $\omega(x)\in N_{F}(x)$, it suffices to show an injection $\psi_x$ from $O_{x}$ to $X\setminus \{x,x^{p}_{F^{o}},\rho(x)\}$, which would imply $|O_{x}|\le d-2$.
 
 The terminal $\rho(x)$ is in $X\setminus O_{x}$, otherwise the pair $\{x,\rho(x)\}\in Y_\alpha$ and $x\in X_\alpha$, a contradiction. Similarly, the projection   $x^{p}_{F^{o}}$ of $x$ onto $F^o$ is in $X\setminus O_{x}$, as it is in $X$ by \eqref{eq:cube-4}, and $O_x$ is a subset of $V(F)$ but $x^{p}_{F^{o}}\in V(F^o)$. Thus  $x, x^{p}_{F^{o}}, \rho(x)\in X\setminus O_{x}$, and $\rho(x)\ne x^{p}_{F^{o}}$ (by \eqref{eq:cube-4}). 
 
 Consider $v\in O_x$. We construct the injection $\psi_x$ as follows:
if $v\in X$ (it is of type 1), map $v$ to $v$.
If $v\notin X$ and $v= z^p_{F}$ for some $z\in X\setminus \{\rho(x)\}$ (it is of type 2), then map $v$ to $z$; here $z\in V(F^o)$. Finally, suppose that $v\notin X$ and $v=\omega(z)$ for some $z\in X'$. Further assume that $v$ is not of type 2 (namely, $v \neq z^p_F$ for any $z \in X\setminus \{\rho(x)\}$), since this case was already considered. Then $v$ is of type 3. Because  $v\notin X$, we have that $z\neq \omega(z)$. From $z\neq \omega(z)$ it follows that $z^p_{F^o}\in X$ (\cref{rmk:cube-3}).  If  $z\neq \rho(x)$, then map $v$ to $z$, else map $v$ to $z^p_{F^o}$. We prove in Claim~\ref{claim:injection-Ox} that the map $\psi_x$ is indeed injective.
\begin{claim}\label{claim:injection-Ox}
Let $\psi_x: O_x\to X\setminus\{x,\rho(x),x^p_{F^o}\}$ be the map defined by
\[\psi_x(v)=\begin{cases}
v,\,\text{if $v\in X$ (type 1)};\\ 
z,\,\text{if  $v\notin X$ and $v= z^p_{F}$ for  $z\in X\setminus \{\rho(x)\}$ (type 2)};\\ 
z,\,\text{if  $v\notin X$, $v=\omega(z)$ for $z\in X'\setminus\{\rho(x)\}$, $v$ is not of type  2 (type 3)} ;\\
%z^p_{F^o},\,\text{if  $v\notin X$, $v=\omega(z)$ for $z\in X'$,  $z= \rho(x)$, and $v$ is not of type 2 (type 3)};\\
\rho(x)^p_{F^o},\,\text{if  $v\notin X$, $v=\omega(\rho(x))$ for $\rho(x)\in X'$, $v$ is not of type 2 (type 3)};
\end{cases}   
\]
Then $\psi_x$ is injective.
\end{claim}

\Cref{fig:5cube-paths} depicts the different types of neighbours of the vertex $s_{2}$ and the injection $\psi_{s_{2}}$ from $O_{s_{2}}$ to $X\setminus \{s_{2},\pi_{F^{o}}(s_{2}),\rho(s_{2})=t_{2}\}$. 

\begin{figure}     
\includegraphics[scale=.9]{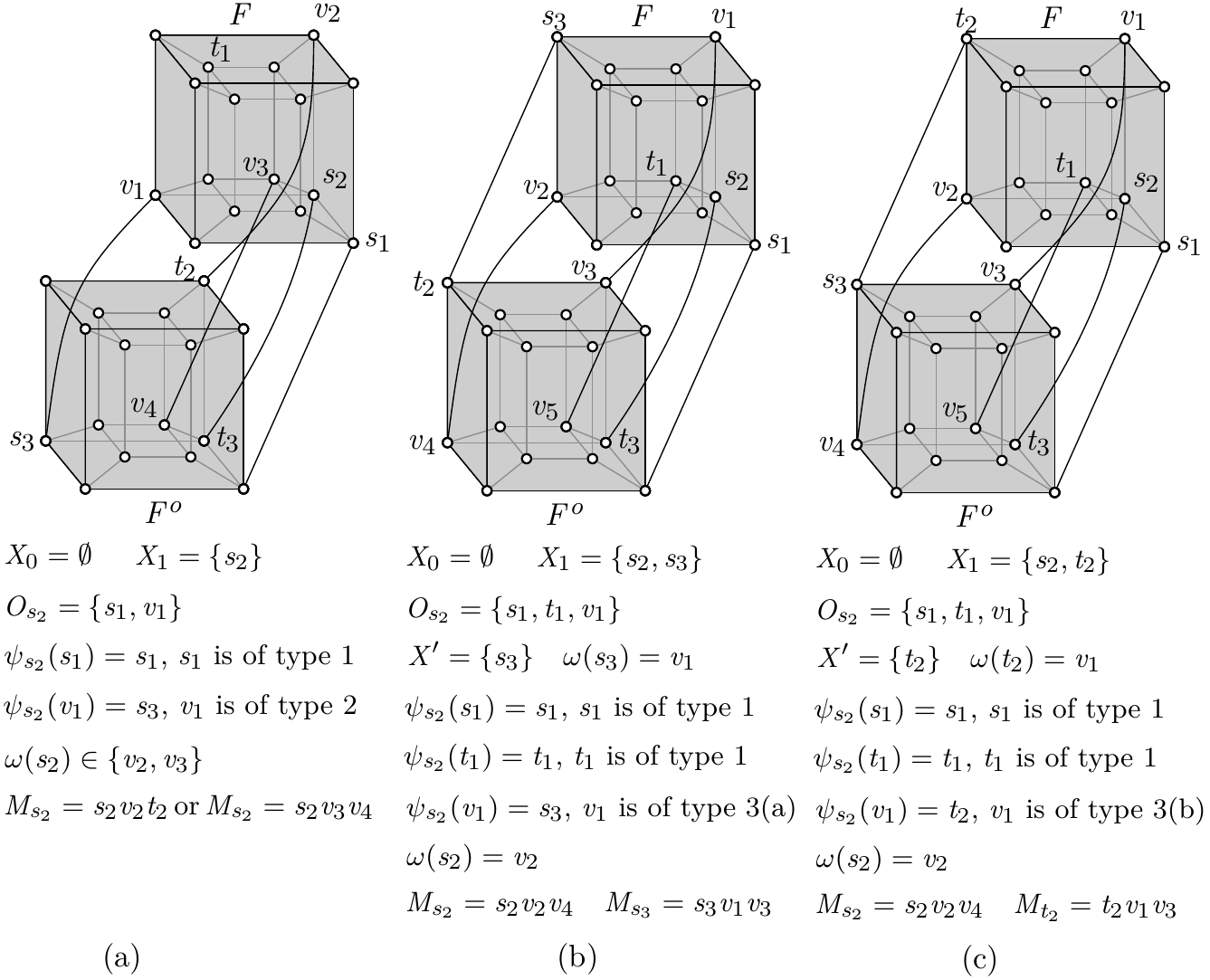}
\caption{Auxiliary figure for the third scenario \cref{thm:cube}.}\label{fig:5cube-paths} 
\end{figure}

\begin{claimproof} For the proof of the claim, we say that a $v$ is of type 3(a) when it satisfies the third line of the definition of $\psi_x$, namely $v\notin X$, $v=\omega(z)$ for some $z\in X'$, $z\neq \rho(x)$, and $v$ is not of type 2. And we say that $v$ is of type 3(b) if it satisfies the fourth line of the definition of $\psi_x$, namely if $v\notin X$, $v=\omega(z)$ for some $z\in X'$,  $z= \rho(x)$, and $v$ is not of type 2. First note the following:
\begin{itemize}
    \item If $v$ is of type 1, then $\psi_x(v)\in X_F\cap N_F(x)$.
    \item If $v$ is of type 2, then $\psi_x(v)\in X\cap V(F^o)$.
    \item if $v$ is of type 3(a), then $\psi_x(v)\in X_F\setminus N_F(x)$ (see \cref{rmk:type3}).
    \item if $v$ is of type 3(b), then $\psi_x(v)\in X\cap V(F^o)$.
\end{itemize}
Assume that $\psi_x(v_1) = \psi_x(v_2) = \gamma$ for $v_1,v_2\in O_x$. Then $\gamma\in X$.

Suppose that $\gamma\in V(F)$. If $\gamma\in N_F(x)$, then both $v_1$ and $v_2$ must be of type 1. In this case, from the definition of $\psi_x$ we conclude that $v_1=v_2=\gamma$. If instead $\gamma\in X_F\setminus N_F(x)$, then $v_1$ and $v_2$ are of type 3(a), and so $v_1 = v_2 = \omega(\gamma)$.

Assume that $\gamma\in V(F^o)$. Then three subcases can occur. If $v_1$ and $v_2$ are both of type 2, then $v_1=v_2=\gamma^p_F$. If instead $v_1$ and $v_2$ are both of type 3(b), then $v_1 = v_2 = \omega(\rho(x))$. Finally suppose that $v_1$ is of type 2 and $v_2$ is of type 3(b). Since $v_1$ is of type 2, we get that  $v_1 = \gamma^p_F$, and since  $v_2$ is of type 3(b) we get that %$v_2= \omega(\rho(x))$,
    %and 
    $\rho(x) =\gamma^p_F$. This implies that $\rho(x) = v_1 \in N_F(x)$, which in turn implies that $\{x,\rho(x)\}\in Y_\alpha$ and $x\in X_\alpha$, contradicting the assumption that $x\notin X_\alpha.$

Therefore, for every $v_1,v_2 \in O_x$, the equality $\psi_x(v_1) = \psi_x(v_2)$ implies that $v_1=v_2$, and the map $\psi_x$ is injective.
\end{claimproof}

\Cref{fig:5cube-paths} also shows the construction of the injection $\psi_{x}$. The existence of the injection $\psi_x$ from $O_{x}$ to $X\setminus \{x,x^{p}_{F^{o}},y\}$ shows that $|O_x|\le d-2$, which yields the existence of the vertex $\omega(x)\in N_{F}(x)$ satisfying Condition~\eqref{eq:injectiveProjection}, and therefore, the injection $\omega$ can be extended to $X'\cup \{x\}$.% of the desired path $M_x = xw_{x}\pi_{F^{o}}(w_x)$. 
 This contradicts the maximality of $X'$ and concludes the proof of the lemma.\end{proof}

For every $x\in X_\beta$, we define the path $M_x = x\omega(x)\pi_{F^o}(\omega(x))$, of length at most 2, from $X_\beta$ to $F^{o}$. The injectivity of $\omega$ and the injectivity of the restriction of $\pi_{F^o}$ to $V(F)$ ensure that the paths $M_x$ are pairwise disjoint. For every $x\in X\cap V(F^o)$ we set $M_x: = x$. Because $\omega$ satisfies Condition~\eqref{eq:injectiveProjection}, the only case when the paths $M_x$ and $M_y$ intersect is when $y=\rho(x)$, which is not a problem.

We now finalise this third scenario. Applying Lemma~\ref{lem:xvalid-path-existence} to $X_\beta$, we get the paths $M_{x}$ from all the terminals in $X_\beta$ to $F^{o}$. We also consider the paths $M_x=x$  for $x\in X\cap V(F^o)$. This implies that we have pairwise disjoint  paths $M_x$ from $X\setminus (X_\alpha\cup\{s_1,t_1\})$ to $F^o$. Denote by $X^o$ the set of vertices in $M_{x}\cap V(F^o)$ for each $x$ in $X\setminus (X_\alpha\cup\{s_1,t_1\})$.  Then
\begin{equation}
\label{eq:cube} |X^o| = |X\setminus (X_\alpha\cup\{s_1,t_1\})| \le d-1.
%|X'|+|X_{1}|+|\pi_{F^{o}}(X_{1})|+|X_{3}|+|Y_{3}|\le 2(k-1)\le d-1.
\end{equation} Let $Y^o$ be the corresponding pairing of the vertices in $X^o$: if $\{x,\rho(x)\}\in Y$ with $x,\rho(x)\in X\setminus (X_\alpha\cup\{s_1,t_1\})$, then the corresponding pair in $Y^o$ is $\{M_{x}\cap V(F^{o}),M_{\rho(x)}\cap V(F^{o})\}$.
 
The induction hypothesis ensures that $F^{o}$ is $(k-1)$-linked. As a consequence, because of \eqref{eq:cube}, $|Y^o|\le (d-1)/2=k-1$, and so there is a $Y^o$-linkage in $F^o$. The $Y^o$-linkage gives the existence of paths $L_{i}^{p}$ in $F^{o}$ between $M_{s_{i}}\cap V(F^{o})$ and $M_{t_{i}}\cap V(F^{o})$ for $s_{i},t_{i} \in X\setminus (X_\alpha\cup\{s_1,t_1\})$. Each path $L_{i}^{p}$ is then extended with the paths $M_{s_{i}}$ and $M_{t_{i}}$ to obtain a path $L_{i}:=s_{i}-t_{i}$ for $s_{i},t_{i} \in X\setminus (X_\alpha\cup\{s_1,t_1\})$. 
  
%   For every pair $(s_i,t_i)$ in $Y_\alpha$, the vertices $s_i,t_i$ are neighbours, therefore are linked. 
   
   It only remains to show the existence of a path $L_{1}:=s_{1}-t_{1}$ in $F$ disjoint from the paths $L_{i}$ for $i\in [2,k]$ (Case (iii)). Suppose that we cannot find a path $L_{1}$ disjoint from the other paths $L_{i}$ with  $i\in [2,k]$. Then there would be a set $S$ in $V(F)$ separating $s_{1}$ from $t_{1}$. The set $S$ would consist of terminal vertices in $X_{F}$ and nonterminal vertices in $\omega(X_\beta)$ (see Lemma \ref{lem:xvalid-path-existence}). % and nonterminal vertices on some path $M_{x}$ for $x\in V(F)\cap (X\setminus (X_\alpha\cup \{s_1,t_1\})$. Each nonterminal vertex in $S$, by definition of the projection operator $M$, amounts to the existence of a terminal vertex in $F^{o}$, namely $x^{p}_{F^{o}}$.
    Since $x\neq \omega(x)$ implies that $\omega(x)\notin X$ and $x^p_{F^o}\in X\setminus \{\rho(x)\}$ (\cref{rmk:cube-3}), we find that, for each such a $\omega(x)$, there is a unique  $x^p_{F^o}\in X\cap V(F^o)$; that is 
\[|S|\le |X_{F} \cup (\omega(X_\beta)\setminus {X})|\leq |X_F| + |X\cap V(F^o)| = |X\setminus \{s_1,t_1\}| = d-1.\]
   By the $(d-1)$-connectivity of $F$, which follows from Balinski's theorem (\cref{thm:Balinski}), the set $S$ would have cardinality $d-1$, which  implies that every terminal in $X_F$ and every nonterminal $\omega(x)$ in $F$ are in $S$. 
   By \cref{prop:cube-cutsets}, the set $S$ would consist of the neighbours of $s_{1}$ or $t_{1}$, say of $s_{1}$, and therefore the vertices in $S$ are pairwise nonadjacent (as there are no triangles in $Q_d$), that is, $\omega(x)=x$ for each $x\in X_\beta$ ($\omega(x)$ is either $x$ or a neighbour of $x$). This implies that $S\subset X$, and so $|S|=|X\setminus \{s_1,t_1\}|=d-1$. Hence, $V(F)\supset S\cup \{s_1,t_1\} = X$, which contradicts the assumptions of \textbf{Scenario 3} that $X\not\subset V(F)$.

The proof of the theorem is now complete.
 \end{proof}
 
 We next illustrate \textbf{Scenario 3} of \cref{thm:cube}.   
 \begin{example}
\label{ex:cube-linkedness} 
  Consider \Cref{fig:5cube-paths}. Then $d=5$, $k=3$, and
 \[\text{$X:=\{s_{1},t_{1},s_{2},t_{2},s_{3},t_{3}\}$ and  $Y:=\{\{s_{1},t_{1}\},\{s_{2},t_{2}\},\{s_{3},t_{3}\}\}$}.\] \cref{lem:xvalid-path-existence} gives us an injection $\omega$ from $(X\setminus \{s_{1},t_{1}\})\cap V(F)$ to $V(F)$ to find $X$-valid paths from the $d-1$ terminals in $(X\setminus \{s_{1},t_{1}\})\cap V(F)$ to $V(F^{o})$. Then we use the $2$-linkedness of $F^{o}$ to find the paths $L_{2}, L_{3}$. We follow the notation of the proof of \cref{thm:cube}.  For each terminal $x\in (X\setminus \{s_{1},t_{1}\})\cap V(F)$,  we define the path $M_{x}:=x\omega(x)\pi_{F^{o}}(\omega(x))$, while for each $x\in X\cap V(F^{o})$ we define the path $M_{x}:=x$. The set $X^{o}$ is the set of vertices in $M_{x}\cap V(F^{o})$. 

 First look at \cref{fig:5cube-paths}(a).  Then $\omega(s_2)\in \{v_2,v_3\}$,  $M_{s_{2}}:=s_{2}v_{2}t_{2}$ or $M_{s_{2}}:=s_{2}v_{3}v_{4}$, say $M_{s_{2}}=s_{2}v_{2}t_{2}$, $M_{t_{2}}:=t_{2}$, $M_{s_{3}}:=s_{3}$, and $M_{t_{3}}:=t_{3}$. It follow that
 \[\text{$X^{o}:=\{t_{2},s_{3},t_{3}\}$ and  $Y^{o}:=\{\{t_{2}\},\{s_{3},t_{3}\}\}$}.\] 
We have paths $L_{2}^{p}:=t_{2}$ and $L_{3}^{p}:=s_{3}-t_{3}$ in $F^{o}$. A path $L_{1}:=s_{1}-t_{1}$ in $F$ should avoid only  $s_{2}$ and $\omega(s_{2})=v_{2}$, and so it exists by the 4-connectivity of $F$.   As a result, the $Y$-linkage in this setting is given by $L_{1}$, $L_{2}:=s_{2}M_{s_{2}}t_{2}$, and $L_{3}:=s_{3}L^{p}_{3}t_{3}$. 

Now look at \cref{fig:5cube-paths}(b). If $\omega(s_3)=v_1$, then $\omega(s_2)=v_2$. Besides,   $M_{s_{2}}:=s_{2}v_{2}v_{4}$, $M_{s_{3}}:=s_{3}v_{1}v_{3}$, $M_{t_{2}}:=t_{2}$, and $M_{t_{3}}:=t_{3}$. It follow that
 \[\text{$X^{o}:=\{v_{3},v_{4},t_{2},t_{3}\}$ and  $Y^{o}:=\{\{v_{4},t_{2}\},\{v_{3},t_{3}\}\}$}.\] 
We have paths $L_{2}^{p}:=v_{4}-t_{2}$ and $L_{3}^{p}:=v_{3}-t_{3}$ in $F^{o}$. A path $L_{1}:=s_{1}-t_{1}$ in $F$ should avoid  $S:=\{s_{2}, s_{3}, \omega(s_{2})=v_{2}, \omega(s_{3})=v_{1}\}$. The path $L_{1}$ exists thanks to \cref{prop:cube-cutsets}, as the set $S$ can separate $s_{1}$ and $t_{1}$ only if it consists of the neighbours of $s_{1}$ or $t_{1}$, which is not possible.   As a result, the $Y$-linkage in this setting is given by $L_{1}$, $L_{2}:=s_{2}M_{s_{2}}v_{4}L^{p}_{2}t_{2}$, and $L_{3}:=s_{3}M_{s_{3}}v_{3}L^{p}_{3}t_{3}$. 

Finally look at \cref{fig:5cube-paths}(c). If $\omega(t_2)=v_1$, then $\omega(s_2)=v_2$. Besides,  $M_{s_{2}}:=s_{2}v_{2}v_{4}$, $M_{s_{3}}:=s_{3}$, $M_{t_{2}}:=t_{2}v_{1}v_{3}$, and $M_{t_{3}}:=t_{3}$. It follow that
 \[\text{$X^{o}:=\{v_{4},s_{3},v_{3},t_{3}\}$ and  $Y^{o}:=\{\{v_{4},v_{3}\},\{s_{3},t_{3}\}\}$}.\] 
We have paths $L_{2}^{p}:=v_{4}-v_{3}$ and $L_{3}^{p}:=s_{3}-t_{3}$ in $F^{o}$. A path $L_{1}:=s_{1}-t_{1}$ in $F$ should avoid  $S:=\{s_{2}, t_{2}, \omega(s_{2})=v_{2}, \omega(t_{2})=v_{1}\}$. The path $L_{1}$ exists thanks to \cref{prop:cube-cutsets}, as the set $S$ can separate $s_{1}$ and $t_{1}$ only if it consists of the neighbours of $s_{1}$ or $t_{1}$, which is not possible.   As a result, the $Y$-linkage in this setting is given by $L_{1}$, $L_{2}:=s_{2}M_{s_{2}}v_{4}L^{p}_{2}v_{3}M_{t_{2}}t_{2}$, and $L_{3}:=s_{3}L^{p}_{3}t_{3}$. \hfill  $\qed$
\end{example}

We are now in a position to answer Wotzlaw's question (\cite[Question 5.4.12]{Ron09}). \cref{thm:cube} in conjunction with \cref{lem:k-linked-subgraph} gives the answer.

\begin{theorem} 
\label{thm:weak-linkedness-cubical} For every $d\ge 1$, a  cubical $d$-polytope is $\floor{d/2}$-linked.  
\end{theorem}  
\begin{proof} Let $P$ be a cubical $d$-polytope. The results for $d=1,2$ are trivial. The case of $d=3$ follows from the connectivity of the graph of $P$ (Balinski's theorem), while the case of $d=4$ follows from \cref{prop:4polytopes}. For $d\ge 5$, since a facet of $P$ is a $(d-1)$-cube with $d-1\ge 4$,  by \cref{thm:cube} it is $\floor{d/2}$-linked. So the $d$-connectivity of the graph of $P$, which follows from Balinski's theorem (\cref{thm:Balinski}), together with \cref{lem:k-linked-subgraph}  establishes the proposition.
\end{proof} 

We improve \cref{thm:weak-linkedness-cubical} in a subsequent paper \cite{BuiPinUgo20a}, where we establish the maximum possible linkedness of $\floor{(d+1)/2}$ for a cubical $d$-polytope with $d\ne 3$. 
 
 \subsection{Strong linkedness of the cube} 
We now show a strong linkedness result for the cube.

\begin{theorem}[Strong linkedness of the cube]\label{thm:cube-strong-linkedness}  For every $d\ge 1$, a $d$-cube is strongly $\floor{d/2}$-linked. 
\end{theorem} 
	
\begin{proof} Let $G$ be the graph of $Q_d$. Let $k:=\floor{d/2}$, and let $X$ be a set of $2k+1$ vertices in the $d$-cube for $d\ge 1$. Arbitrarily pair $2k$ vertices in $X$ to obtain $Y:=\{\{s_{1},t_{1}\},\ldots,\{s_{k},t_{k}\}\}$.  Let $x$ be the vertex of $X$ not being paired in $Y$. We aim to find a $Y$-linkage $\{L_{1},\ldots, L_{k}\}$ where each path $L_{i}$ joins the pair $\{s_{i},t_{i}\}$ and avoids the vertex $x$. 

The theorem is trivially true for the cases $d=1,2$. So assume that $d\ge 3$.

Suppose that $d=2k+1$. In this case, the result follows from Balinski's theorem and \cref{thm:cube}. For $d=3$, we have that $k=1$, and so the 3-connectivity of $G$ (Balinski's theorem) ensures that we can find a path $L_1=s_1-t_1$ that avoids $x$. For $d\ge 5$, we let $y$ be a vertex of $G$ not in $X$. Then, since $G$ is $(k+1)$-linked (\cref{thm:cube}), we can find $k+1$ disjoint $L_1,\ldots,L_k,L_{k+1}$ paths such that $L_i=s_i-t_i$, for $i\in [1,k]$, and $L_{k+1}:=x-y$. It is now plain that the linkage $Y$ in $G-x$ gives that $G-x$ is $k$-linked.

Suppose that $d=2k$. The result for $d=4$ is given by \cref{thm:4polytopes-strong-linkedness}. So assume that $d\ge 6$. 

From \cref{lem:facets-association} it follows that there exists a pair $\{F,F^o\}$ of opposite facets of $Q_{d}$ that is not associated with $X_{x}:=X\setminus \{x\}$, since $|X\setminus \{x\}|=d$ and there are $d$  pairs $\{F,F^o\}$ of opposite facets in $Q_{d}$. Assume $x\in V(F^{o})$.  Let $X^{p}:=\pi_{F}(X_{x})$; that is, the set $X^{p}$ comprises the vertices in $X_{x}\cap V(F)$ plus the projections of $X_{x}\cap V(F^{o})$ onto $F$. Denote by $Y^{p}$ the corresponding pairing of the vertices in $X^{p}$; that is, $Y^{p}:=\{\{\pi_{F}(s_{1}), \pi_{F}(t_{1})\},\ldots, \{\pi_{F}(s_{k}), \pi_{F}(t_{k})\}\}$. Then $|X^{p}|=d$ and $|Y^{p}|=k$.  Find a $Y^{p}$-linkage $\{L_{1}^{p},\ldots, L_{k}^{p}\}$ in $F$ with $L_{i}^{p}:=\pi_{F}(s_{i})-\pi_{F}(t_{i})$ by resorting to the $k$-linkedness of $F$ (\cref{thm:cube}).  Adding $s_{i}\in V(F^{o})$ or $t_{i}\in V(F^{o})$ to the path $L_{i}^{p}$, if necessary, we extend the linkage $\{L_{1}^{p},\ldots, L_{k}^{p}\}$ to the required $Y$-linkage. \end{proof}

\section{Linkedness inside the cube}

The \textit{boundary complex} of a polytope $P$ is the set of faces of $P$ other than $P$ itself. And the \textit{link} of a vertex $v$ in a polytope $P$, denoted $\lk(v,P)$, is the set of faces of $P$ that do not contain $v$ but lie in a facet of $P$ that contains $v$ (\cref{fig:4-cube}). According to \cite[Ex.~8.6]{Zie95}, the link of a vertex in a $d$-polytope is combinatorially equivalent to the boundary complex of a $(d-1)$-polytope; in particular, for $d\ge 3$ the graph of the link is isomorphic to the graph of a $(d-1)$-polytope. It follows that the link of a vertex of $Q_{d+1}$ is combinatorially equivalent to the boundary complex of a cubical $d$-polytope.

\begin{remark}
\label{rmk:opposite-vertex}
The link of a vertex $v$ in a $d$-cube $Q_{d}$ is obtained by removing all the faces of $Q_{d}$ that contain $v$ or $ v^{o}$ (\cref{fig:4-cube}).  \end{remark} 

\begin{figure}
\includegraphics{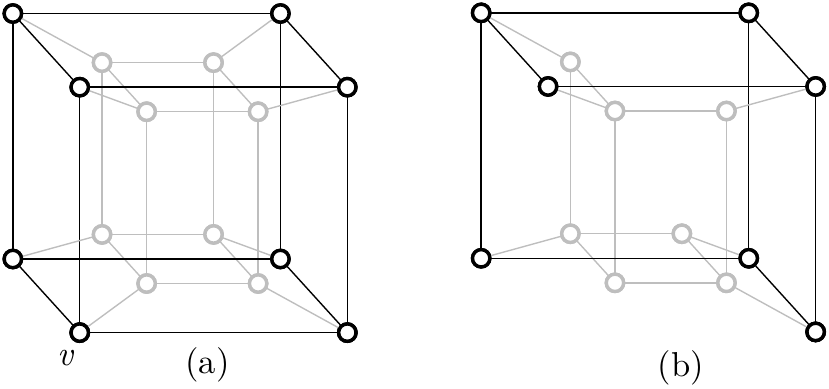} 
\caption{(a) The 4-cube with a vertex $v$ highlighted. (b) The link of the vertex $v$.}\label{fig:4-cube} 
\end{figure}

We verify that, for every $d\ge 2$ such that $ d\ne 3$, the link of a vertex in a $(d+1)$-cube is $\floor{(d+1)/2}$-linked.
 
\begin{proposition}
\label{prop:link-cubical} For every $d\ge 2$ such that $ d\ne 3$, the link of a vertex in a $(d+1)$-cube $Q_{d+1}$ is $\floor{(d+1)/2}$-linked.  
\end{proposition}
\begin{proof} 
The proposition trivially holds for the case $d=2$, and the case $d= 4$ is given by \cref{prop:4polytopes}. So assume that $d\ge 5$

Let $k:=\floor{(d+1)/2}$. Then $k\ge 3$. Let $v$ and $v^{o}$ be opposite vertices of $G(Q_{d+1})$; that is, $\dist_{Q_{d+1}}(v,v^{o})=d+1$. Let $X$ be a given set of $2k$ vertices in $\lk(v,Q_{d+1})$ (\cref{rmk:opposite-vertex}), and let $Y:=\{\{s_{1},t_{1}\},\ldots,\{s_{k},t_{k}\}\}$ be an arbitrary pairing of the vertices in $X$. We show that there exists a $Y$-linkage $\{L_1,\ldots,L_k\}$ in $\lk(v,Q_{d+1})$ where  each path $L_{i}$ joins the pair $\{s_{i},t_{i}\}$ and avoids the vertices $v$ and $v^o$.
 
Since $|X|-1\le d$ and there are $d+1$ pairs of opposite facets in $Q_{d+1}$, from \cref{lem:facets-association} there exists a pair $\{F,F^o\}$ of opposite facets of $Q_{d+1}$ that is not associated with $X$. This means that, for every $x\in X\cap V(F)$, its projection $\pi_{F^o}^{Q_{d+1}}(x)\not \in X$, and that, for every $x\in X\cap V(F^o)$, its projection  $\pi_{F}^{Q_{d+1}}(x)\not \in X$. Henceforth we write $\pi_{F}$ rather than $\pi_{F}^{Q_{d+1}}$. Assume that $v\in F$ and $v^{o}\in F^o$. We consider two cases based on the number of terminals in the facet $F$.%; for the sake of readability, the third case is in turn decomposed into two subcases highlighted in bold. 

In what follows, we implicitly use the $d$-connectivity of $F$ or $F^{o}$ for $d\ge5$. 
 
\begin{case} $|X\cap V(F)|=d+1$.\end{case}

Since $F$ is a $d$-cube, it is $\floor{(d+1)/2}$-linked by \cref{thm:cube}, and hence, we can find $k$ pairwise disjoint paths $\bar L_{1},\ldots,\bar L_{k}$ in $F$ between $s_{i}$ and $t_{i}$ for each $i\in[1,k]$. If no path $L_{i}$ passes through $v$, we let $L_i:=\bar L_i$ for $i\in [1,k]$, and so $\{L_1,\ldots, L_k\}$ is the desired $Y$-linkage. So suppose one of those paths, say $\bar L_{1}$, passes through $v$; there can be only one such path. 

In this case, we consider the two neighbours $w_1$ and $w_2$ of $v$ on $\bar L_1$ so that $v\notin w_1\bar L_1s_1$ and $v\notin w_2\bar L_1t_1$. Since $\dist_{Q^{d+1}}(v,v^{o})= d+1\ge 3$, we have that $\pi_{F^o}(w_1)\ne v^o$ and $\pi_{F^o}(w_2)\ne v^o$. Thus, we find a $\pi_{F^{o}}(w_1)-\pi_{F^{o}}(w_{2})$  path $M_{1}$ in $F^o$ that avoids $v^{o}$. So $L_{1}$ then becomes $s_{1}\bar L_1w_1\pi_{F^{o}}(w_1)M_{1}\pi_{F^{o}}(w_{2})w_2\bar L_1t_{1}$. Finally, we let $L_i:=\bar L_i$ for $i\in [2,k]$, and so $\{L_1,\ldots, L_k\}$ is the desired $Y$-linkage.

%If neither the projection of $s_{1}$ onto $F^{o}$ nor the projection of $t_{1}$ onto $F^{o}$ is $v^{o}$, then find a $\pi_{F^{o}}(s_{1})-\pi_{F^{o}}(t_{1})$ path  $\bar L_{1}$ in $F^{o}$ that avoids $v^{o}$. So $L_{1}$ would then become $s_{1}\pi_{F^{o}}(s_{1})\bar L_{1}\pi_{F^{o}}(t_{1})t_{1}$. If the projection of either $s_{1}$ or  $t_{1}$ onto $F^{o}$ is $v^{o}$, say that of $s_{1}$, then, since $\dist_{Q^{d+1}}(v,v^{o})= d+1\ge 5$ and $\dist_{Q^{d+1}}(s_{1},v)= d\ge 4$, there must be a neighbour $w$ of $s_{1}$ on $L_{1}$ that is different from $v$. 
By symmetry, the proposition also holds if  $|X\cap V(F^{o})|=d+1$. 

\begin{case} $|X\cap V(F^o)| \le |X\cap V(F)|\le d$.\end{case}

It follows that $|X\cap V(F^o)|\le \floor{(d+1)/2}$.
If some terminal $x$ in $F^o$ is adjacent to $v$: $\pi_{F}(x)= v$, without loss of generality, assume that it is  $t_1$. In any case, there is at least one terminal in $F^o$, and we may assume that it is $t_1$. The facet $F$ is a $d$-cube, and so it is $\floor{(d+1)/2}$-linked by \cref{thm:cube}.

%Let $X^{p}:=\pi_{F}(X\setminus \{t_1\})$. The set $X^{p}$ comprises the terminals in $X\cap V(F)$ together with the  projections onto $F$ of the vertices in $X\cap (V(F^o)\setminus \{t_1\})$. Then \[X^{p}=\{\pi_F(s_1),\pi_F(s_2),\pi_F(t_2),\ldots, \pi_F(s_k),\pi_F(t_k)\}.\]

The $k$-linkedness of $F$ ensures that in $F$ there are $k$ disjoint paths $M_{1}:=\pi_{F}(s_{1})-v$ and  $\bar L_{i}:=\pi_{F}(s_{i})-\pi_{F}(t_{i})$ for $i\in [2,k]$. For each $i\in [2,k]$, the path $\bar L_i$ lies in $F$, is $X$-valid,  and avoids $v$ (and also $v^o\in F^o$). Besides,  for each $j\in [2,k]$, each path $\bar L_{j}$  extends to a path $L_{j}:=s_{j}-t_{j}$, if necessary. It remains to find an $X$-valid path $L_1:=s_1-t_1$ that avoids $v$ and $v^o$.   

Let $S:=\{v^{o}\}\cup ((X\cap V(F^{o}))\setminus \{t_{1}\})$.  Then, by assumption $|X\cap V(F^{o})|\le (d+1)/2$, and so $|S|\le (d+1)/2-1+1\le d-1$ for $d\ge 3$.

 Let $w\in V(F)$ be a neighbour of $v$ on $M_1$. This neighbour exists, since $\pi_F(s_1)\neq v$, and so the length of $M_1$ is at least 1. The projection $\pi_{F^o}(w)$ of $w$ onto $F^o$ is not in $ X\setminus\{s_1\}$ because $w\notin X^p\setminus \{\pi_F(s_1)\}$ and the path $M_1$ is disjoint from the paths $\bar{L}_i$, $i\in [2,k]$. By the $d$-connectivity of $F^o$ we can find a path $\bar L_1$ in $F^o$ from $\pi_{F^o}(w)$ to $t_1$ that avoids $S$. Hence the path $L_{1}$ then becomes $s_{1}\pi_F(s_1)M_1w\pi_{F^{o}}(w)\bar L_{1}t_{1}$. The set $\{L_1,\ldots,L_k\}$ is the desired $Y$-linkage.

This completes the proof of the case and of the proposition.   
\end{proof}

\section{Acknowledgements} The authors would like to thank the anonymous referees for their comments and suggestions. The presentation of the paper has greatly benefited from their input.

%---REFERENCES---

\providecommand{\bysame}{\leavevmode\hbox to3em{\hrulefill}\thinspace}
\providecommand{\MR}{\relax\ifhmode\unskip\space\fi MR }
\providecommand{\MRhref}[2]{%
  \href{http://www.ams.org/mathscinet-getitem?mr=#1}{#2}
}
\providecommand{\href}[2]{#2}

\end{document}